\documentclass[11pt,reqno]{amsart}
\usepackage[margin=8em, footskip=2em]{geometry}
\usepackage[latin1]{inputenc}
\usepackage[T1]{fontenc}
\usepackage[american]{babel}
\usepackage{amsmath,amsfonts,amssymb,amsthm,bm,wasysym,pifont,mathtools}
\usepackage{tikzsymbols}
\usepackage{lastpage}
\usepackage{enumitem}
\usepackage{xcolor}
\usepackage[colorlinks,linkcolor=blue!50!black,citecolor=blue!50!black,pagebackref,hypertexnames=false, breaklinks]{hyperref}
\usepackage{cleveref}
\usepackage{float,caption}
\usepackage{tikz}
\usetikzlibrary{matrix,arrows,backgrounds,calc,chains,automata,positioning,patterns}
\usetikzlibrary{decorations,decorations.pathmorphing}
\usetikzlibrary{shapes.geometric,calc}
\allowdisplaybreaks
\newtheorem{theorem}{Theorem}
\newtheorem{proposition}[theorem]{Proposition}

\theoremstyle{definition}

\newtheorem*{def*}{Definition}

\newtheorem{remark}{Remark}

\newtheorem*{thm*}{Theorem}
\newtheorem*{lem*}{Lemma}
\newtheorem*{prop*}{Proposition}
\newtheorem*{rem*}{Remark}

\def\SG{{\rm SG}}
\numberwithin{equation}{section}
\title[QE on SG]{
Dirac measures can be quantum limits \\
on the Sierpinski gasket}
\date{\today}
\author{Patricia Alonso Ruiz}
\address{Institute of Mathematics, Friedrich Schiller University, Jena, Germany}
\email{patricia.alonso.ruiz@uni-jena.de}
\author{Nikolay Tzvetkov}
\address{UMPA, \'Ecole Normal Sup\'erieure Lyon, Lyon, France}
\email{nikolay.tzvetkov@ens-lyon.fr}
\subjclass[2010]{58J51;28A80}
\keywords{quantum ergodicity; Laplacian; eigenfunctions; Sierpinski gasket}
\begin{document}
\begin{abstract}
    We prove that on the Sierpinski gasket, a prototype of compact fractal space, Dirac distributions arise as the weak limit of probability measures associated with sequences of high energy eigenfunctions of the Laplacian in a way that cannot happen in compact Riemannian manifolds.
\end{abstract}
\maketitle
\section{Introduction}

Given an orthonormal basis $\{\psi_k\}_{k\geq 1}$ of eigenfunctions of the Laplace operator (with suitable boundary conditions) on a compact domain $K$ equipped with a (probability) measure $\mu$ (the one involved in the definition of the Laplacian), the expression $|\psi_k|^2d\mu$ naturally defines another probability measure on $K$.  Broadly speaking,  quantum mechanics understands it as the probability distribution of a particle's position at the energy level described by the corresponding eigenvalue.

\medskip

A sequence of eigenfunctions $\{\psi_n\}_{n\geq 1}$ whose associated eigenvalues $\lambda_{n}$ increase towards infinity is said to \emph{equidistribute} in position if the sequence of measures $\{|\psi_n|^2d\mu\}_{n\geq 1}$ converges weakly to the uniform measure. In other words, if
\begin{equation}\label{E:def_equidistribute}
    \int_K\eta|\psi_n|^2d\mu\xrightarrow{n\to\infty}\int_K\eta\,d\mu
\end{equation}
for all $\eta\in C(K)$. 
This kind of result classically belongs to the research area of quantum ergodicity, a subject with an extensive literature; for further details, we refer to the review articles~\cite{Sar11,Zel19, Dya23} and to the books~\cite{Zwo12,Ana22}.

\medskip

A fundamental question in quantum ergodicity is, whether something else than~\eqref{E:def_equidistribute} can happen, and if so, what kind of measures can arise as limits of such sequences of high energy eigenfunctions.

\medskip

The main result in this note is that, on the Sierpinski gasket, which is a prototype fractal compact space, for any point in a dense set of the gasket it is possible to find a sequence of high energy eigenfunctions that converges to a Dirac measure concentrated at that point. This phenomenon is in stark contrast to Euclidean models like the disk, or to other compact Riemannian manifolds with and without boundary. 
Indeed, in the setting of Riemannian manifolds Dirac measures cannot be limits of a sequence of high energy eigenfunctions because it would be in contradiction with a version of the propagation of singularity theorem, see \cite{Ho,Zwo12} for the boundaryless case and \cite{MS1,MS2,G,BL,Sun} for the more delicate case of manifolds with boundaries. 
\medskip

The paper is organized as follows: Section~\ref{S:Background} provides the necessary background on the Sierpinski gasket and the eigenfunctions of its standard Laplacian. In Section~\ref{S:main_result} we prove the main result, Theorem~\ref{T:main_result}, which roots in the existence of a rich family of localized eigenfunctions.

\section{Background and preliminaries}\label{S:Background}
\subsection{The Sierpinski gasket (SG)}\label{SS:SG_basics}
The standard Sierpinski gasket ($\SG$), see e.g. Figure~\ref{F:SG_measure}, is the unique compact subset of $\mathbb{R}^2$ that arises as the fixed point
\begin{equation}\label{E:SG_fixed_point}
    \SG=\bigcup_{k=0}^2F_k(\SG)
\end{equation} 
where $F_k\colon\mathbb{R}^2\to\mathbb{R}^2$, $k=0,1,2$, are given by
\begin{equation}\label{E:def_Fi}
    F_k(p):=\frac{1}{2}(p-p_k)+p_k
\end{equation}
and $V_0:=\{p_0,p_1,p_2\}$ denotes the set of vertices of an equilateral triangle of side length one, c.f. Figure~\ref{F:SG_measure}. The set $V_0$ is regarded as the natural boundary of $\SG$. Due to its self-similar structure, for each $m\geq 1$, one can decompose $\SG$ into $3^m$ rescaled copies of itself, $F_w(\SG)$, where $F_{w}:=F_{w_1}{\circ}\cdots \circ F_{w_m}$ with $w=w_1\ldots w_m\in\{0,1,2\}^m$.

\medskip

As a metric measure space, $\SG$ is equipped with the Euclidean metric and the standard self-similar Bernoulli measure $\mu$ that gives the same weight to each copy $F_w(\SG)$. This measure thus satisfies
\begin{equation}\label{E:measure_scaling}
    \mu(F_w(B))=3^{-m}\mu(B),
\end{equation}
for any $B\subseteq\SG$ and $m\geq 0$, and it is also comparable to the natural $d_H$-dimensional Hausdorff measure on $\SG$, where $d_H=\frac{\log 3}{\log 2}$ is the (Euclidean) Hausdorff dimension of $\SG$.

\begin{figure}[H]
    \includegraphics[scale=.45]{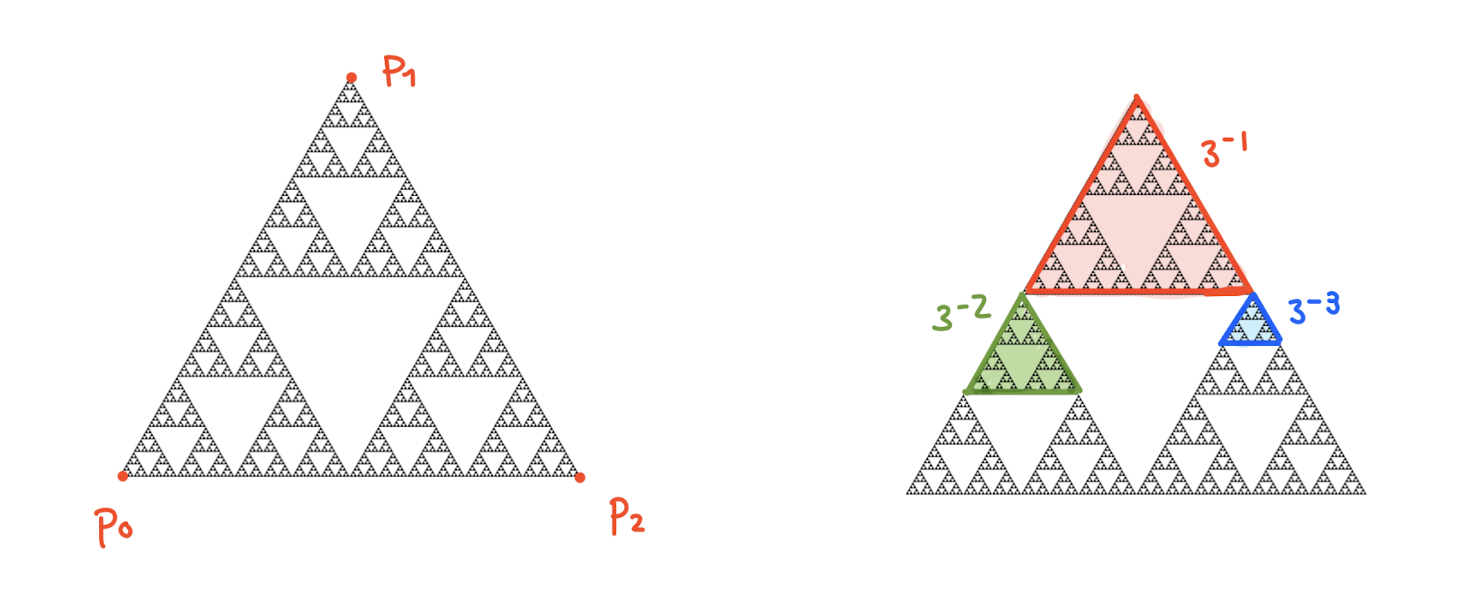}
    \caption{The standard Sierpinski gasket $\SG$, boundary points and the measure of different cells.}
    \label{F:SG_measure}
\end{figure}

\subsection{The Laplace operator on SG} 
The standard intrinsic Laplace operator on $\SG$ is usually constructed through finite graph approximations, based on the fact that the sets $V_m:=\bigcup_{w\in\{0,1,2\}^m}F_{w}(V_0)$ approximate $\SG$ in the sense that $V_m\subseteq V_{m+1}$ and
\begin{equation}\label{E:Vstar}
V_*:=\bigcup_{m\geq 0}V_m
\end{equation}
is dense in $\SG$. For each $m\geq 0$, one thus considers a graph with vertex set $V_m$ that describes the finite $m$-level approximation of $\SG$ as suggested in Figure~\ref{F:SG_approx}.
\begin{figure}[H]
\begin{center}
\renewcommand{\arraystretch}{0.5}
\begin{tabular}{cccc}
\begin{tikzpicture}
\tikzstyle{every node}=[draw,circle,fill=black,minimum size=2pt, inner sep=0pt]`'
\draw ($(0:0)$) node () {} --++ ($(0:2)$) node () {} --++ ($(120:2)$) node () {} --++ ($(240:2)$) node () {};
\end{tikzpicture}
\hspace*{2em}&
\begin{tikzpicture}
\tikzstyle{every node}=[draw,circle,fill=black,minimum size=2pt, inner sep=0pt]
\draw ($(0:0)$) node () {} --++ ($(0:2/2)$) node () {} --++ ($(120:2/2)$) node () {} --++ ($(240:2/2)$) node () {};
\draw ($(0:2/2)$) node () {} --++ ($(0:2/2)$) node () {} --++ ($(120:2/2)$) node () {} --++ ($(240:2/2)$) node () {};
\draw ($(60:2/2)$) node () {} --++ ($(0:2/2)$) node () {} --++ ($(120:2/2)$) node () {} --++ ($(240:2/2)$) node () {};
\end{tikzpicture}
\hspace*{2em}&
\begin{tikzpicture}
\tikzstyle{every node}=[draw,circle,fill=black,minimum size=2pt, inner sep=0pt]
\draw ($(0:0)$) node () {} --++ ($(0:2/4)$) node () {} --++ ($(120:2/4)$) node () {} --++ ($(240:2/4)$) node () {};
\draw ($(0:2/4)$) node () {} --++ ($(0:2/4)$) node () {} --++ ($(120:2/4)$) node () {} --++ ($(240:2/4)$) node () {};
\draw ($(60:2/4)$) node () {} --++ ($(0:2/4)$) node () {} --++ ($(120:2/4)$) node () {} --++ ($(240:2/4)$) node () {};
\foreach \a in {0,60}{
\draw ($(\a:2/2)$) node () {} --++ ($(0:2/4)$) node () {} --++ ($(120:2/4)$) node () {} --++ ($(240:2/4)$) node () {};
\foreach \b in{0,60}{
\draw ($(\a:2/2)+(\b:2/4)$)node () {} --++ ($(0:2/4)$) node () {} --++ ($(120:2/4)$) node () {} --++ ($(240:2/4)$) node () {};
}
}
\end{tikzpicture}
\hspace*{2em}&
\begin{tikzpicture}
\tikzstyle{every node}=[draw,circle,fill=black,minimum size=2pt, inner sep=0pt]
\draw ($(0:0)$) node () {} --++ ($(0:2/8)$) node () {} --++ ($(120:2/8)$) node () {} --++ ($(240:2/8)$) node () {};
\draw ($(0:2/8)$) node () {} --++ ($(0:2/8)$) node () {} --++ ($(120:2/8)$) node () {} --++ ($(240:2/8)$) node () {};
\draw ($(60:2/8)$) node () {} --++ ($(0:2/8)$) node () {} --++ ($(120:2/8)$) node () {} --++ ($(240:2/8)$) node () {};
\foreach \a in {0,60}{
\draw ($(\a:2/4)$) node () {} --++ ($(0:2/8)$) node () {} --++ ($(120:2/8)$) node () {} --++ ($(240:2/8)$) node () {};
\foreach \b in{0,60}{
\draw ($(\a:2/4)+(\b:2/8)$)node () {} --++ ($(0:2/8)$) node () {} --++ ($(120:2/8)$) node () {} --++ ($(240:2/8)$) node () {};
}
}
\foreach \c in{0,60}{
\draw ($(\c:2/2)$) node () {} --++ ($(0:2/8)$) node () {} --++ ($(120:2/8)$) node () {} --++ ($(240:2/8)$) node () {};
\draw ($(\c:2/2)+(0:2/8)$) node () {} --++ ($(0:2/8)$) node () {} --++ ($(120:2/8)$) node () {} --++ ($(240:2/8)$) node () {};
\draw ($(\c:2/2)+(60:2/8)$) node () {} --++ ($(0:2/8)$) node () {} --++ ($(120:2/8)$) node () {} --++ ($(240:2/8)$) node () {};
\foreach \a in {0,60}{
\draw ($(\c:2/2)+(\a:2/4)$) node () {} --++ ($(0:2/8)$) node () {} --++ ($(120:2/8)$) node () {} --++ ($(240:2/8)$) node () {};
\foreach \b in{0,60}{
\draw ($(\c:2/2)+(\a:2/4)+(\b:2/8)$)node () {} --++ ($(0:2/8)$) node () {} --++ ($(120:2/8)$) node () {} --++ ($(240:2/8)$) node () {};
}
}
}
\end{tikzpicture}
\\ [1em]
$V_0$ \hspace*{2em}& $V_1$ \hspace*{2em}& $V_2$ \hspace*{2em}& $V_3$
\end{tabular}
\end{center}
\caption{Graph approximations of the Sierpinski gasket.}
\label{F:SG_approx}
\end{figure}
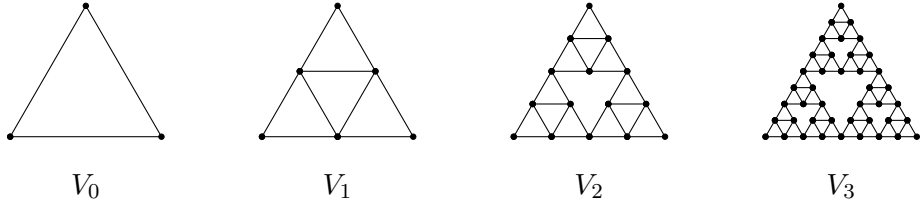

The graph energy associated with the $m$-level approximation of $\SG$ is given by
\begin{equation}\label{E:graph_energy_m}
\mathcal{E}_m (u):=\frac{1}{2}\sum_{p\stackrel{m}{\sim} q}(u(q)-u(p))^2,
\end{equation}
for any $u\colon V_m\to \mathbb{R}$, where $p\stackrel{m}{\sim}q$ means that $p,q\in V_m$ are neighbors as indicated in Figure~\ref{F:SG_approx}. Noting that $V_*:=\bigcup_{m\geq 0}V_m$ is dense in $\SG$, the standard energy on $\SG$ arises as the limit
\begin{equation*}
    \mathcal{E}(u):=\lim_{m\to\infty}\left(\frac{5}{3}\right)^m\mathcal{E}_m (u|_{V_m}),
\end{equation*}
for any $u\in C(\SG)$, which is extended to a bilinear form $\mathcal{E}(u,v)$ via polarization, see e.g.~\cite[Section 1.4]{Str06}. The ``Neumann domain'' of $\mathcal{E}$, denoted by $\mathcal{F}$, is obtained as the closure of continuous functions with respect to $(\mathcal{E}(u)+\|u\|_{L^2})^{1/2}$. The Dirichlet domain $\mathcal{F}_0$, consists of those functions in $\mathcal{F}$ that are zero on the boundary $V_0$. For the ease of the presentation, we will consider here the Neumann case only, and note that the same result can be obtained in the Dirichlet case.

\medskip

The bilinear form $(\mathcal{E},\mathcal{F})$ is in fact a local and regular Dirichlet form~\cite[Theorem 3.4.6]{Kig01} and has thus an associated infinitesimal generator $\Delta$, which is a non-positive self-adjoint operator. In the case of $\SG$, it has pure point spectrum and only accumulation point at $-\infty$, see e.g.~\cite[Theorem 4.2]{FS92}, and is regarded as the standard intrinsic Laplace operator on $\SG$ with Neumann boundary conditions. More precisely, we say that $-\Delta u=f$ for $u\in\mathcal{F}$, if and only if
    \begin{equation*}
        \mathcal{E}(u,v)=\langle f,v\rangle\qquad\forall\,v\in\mathcal{F},
    \end{equation*}
where $\langle \cdot,\cdot\rangle$ denotes the standard inner product in $L^2_\mu(\SG)$. 

\medskip

In addition, it is also possible to give a pointwise expression of the Laplacian in terms of graph approximations as the renormalized limit 
\begin{equation*}
    \Delta u=\lim_{m\to\infty}
    5^m\Delta_m (u|_{V_m}),
\end{equation*}
where $\Delta_m$ denotes the graph Laplacian associated with the approximating graph at level $m$
\begin{equation}\label{E:def_graph_Laplacian_m}
    \Delta_mu(p)=\sum_{q\stackrel{m}{\sim}p}\big(u(p)-u(q)\big)
\end{equation}
for any $u\colon V_m\to\mathbb{R}$, see e.g.~\cite[Example 3.7.3]{Kig01} or~\cite[Section 2.2]{Str06}. 
\subsection{Eigenvalues of the Laplacian on SG}\label{SS:Eigenvalues}
An early observation regarding the analytic properties of the Laplacian on $\SG$ was that its spectrum could be described by means of a so-called \emph{spectral decimation method}, introduced in~\cite{RT82} and further studied by Fukushima and Shima in~\cite{FS92}. In short, an eigenvalue $\lambda$ of $-\Delta$ arises as the limit 
\begin{equation}\label{E:ev_as_lim_a}
\lambda=\lim_{m\to\infty}5^m\lambda_m
\end{equation}
for a unique sequence $\{\lambda_m\}_{m\geq j}$, where $\lambda_m$ is an eigenvalue of the graph Laplacian $-\Delta_m$. The sequence starts at a specific level $j\geq 1$ see~\cite[Theorem 5.1]{FS92}, that is referred to as the \emph{generation of birth} of the eigenvalue $\lambda$. 
In Neumann case, such ``first ancestors'' $\lambda_j$ can only take the values $5$ or $6$, and one often calls an eigenvalue $\lambda$ either a $5$-series eigenvalue, or a $6$-series eigenvalue, depending on which was the first element in the sequence. In the Dirichlet case, also the value $2$ is allowed, but this does not affect the construction relevant to the main result here, which will only involve eigenfunctions associated with $6$-series eigenvalues.

\medskip

For that purpose, we will denote by $\Lambda_j$, $j\geq 2$, the \emph{lowest} $6$-series eigenvalue of $-\Delta$ with generation of birth $j$. About these eigenvalues it is known that 
\begin{equation}\label{R:lowest_D}
\Lambda_j=5^{j-2}\Lambda_2\qquad\text{and}\qquad{\rm mult}(\Lambda_j):=\frac{1}{2}(3^{j}-3)=:N_j
\end{equation}
for all $j\geq 2$, c.f.~\cite[Section 5]{FS92} or~\cite[Section 3.3]{Str06}. Notice that $N_j$ is the number of vertices in $V_{j-1}{\setminus}V_0$.

\subsection{Eigenfunctions of the Laplacian on SG}\label{SS:Eigenfunctions}
Given a specific eigenvalue $\lambda$ of $-\Delta$, there is a standard procedure to construct a particular 
eigenfunction associated with it. The method consists in applying the so-called \emph{eigenfunction extension algorithm}~\eqref{E:efct_extension} together with a density argument. We briefly review the method for a $6$-series eigenvalue, since these will be the ones relevant for our purposes, although it works analogously for any other eigenvalue.
\medskip

Let thus $\lambda=\Lambda_j$ be a fixed 6-series eigenvalue with generation of birth $j\geq 2$ 
and describe an eigenfunction $\phi$ associated with this eigenvalue.

\medskip

Since $\lambda$ is a $6$-series, the first element in the approximating sequence is $\lambda_j=6$. The latter is an eigenvalue of the graph Laplacian $-\Delta_j$. With some linear algebra one can construct an associated eigenvector, which we regard as a function $\phi_j\colon V_j\to\mathbb{R}$. As is customary, the construction is done in a way that $\phi_j$ takes the value 2 at one of the vertices in $V_{j}{\setminus}V_0$. A possible choice for $j=2$ is described in Figure~\ref{F:6series_ev_level2}, for more details we refer e.g.~\cite[Section 3.3]{Str06}.
    
\begin{figure}[H]
    \includegraphics[scale=.35]{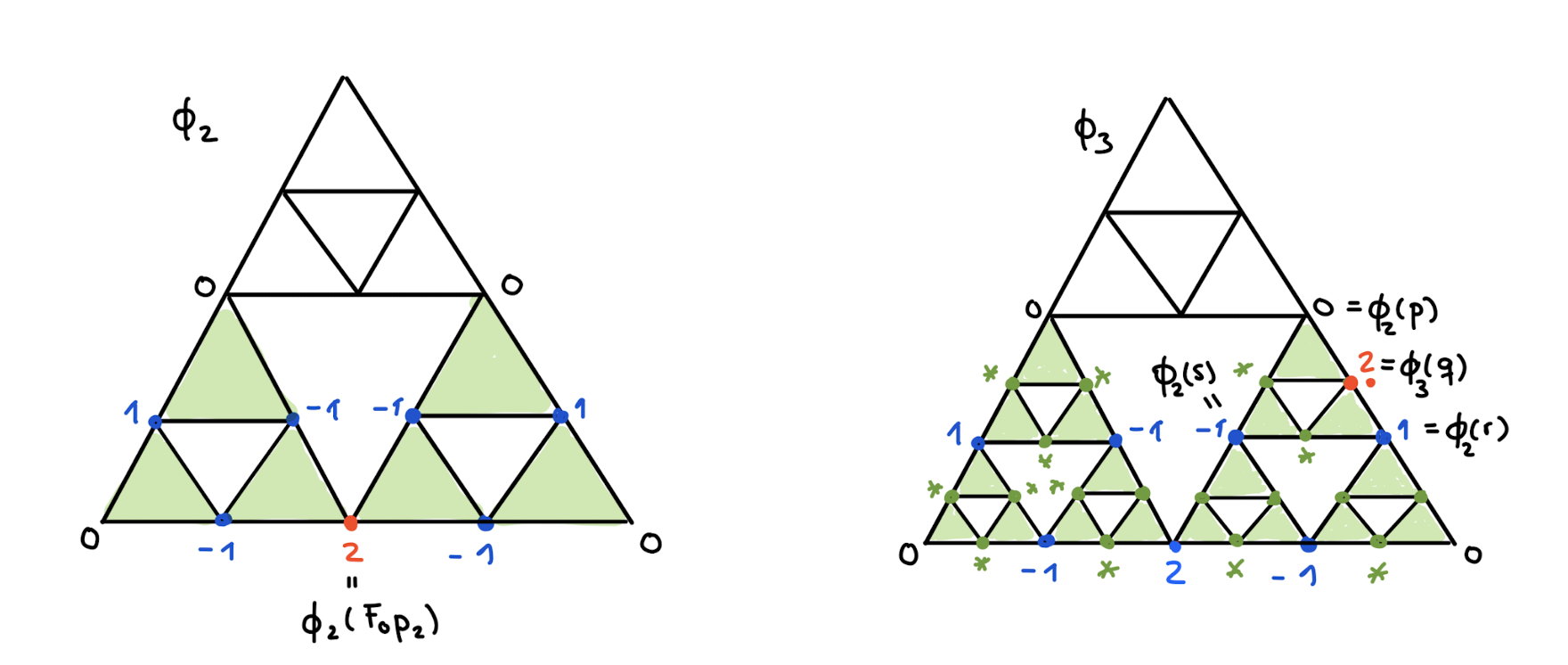}
    \caption{Eigenvector $\phi_2\colon V_2\to\mathbb{R}$ associated with the eigenvalue graph Laplacian $-\Delta_2$ and its extension to (an eigenvector) $\phi_3\colon V_3\to\mathbb{R}$ at level 3.}
    \label{F:6series_ev_level2}
\end{figure}

 
\medskip

Next, one extends $\phi_j$ to a function in the next approximation level, $\phi_{j+1}\colon V_{j+1}\to\mathbb{R}$, which will in fact be an eigenvector of $\lambda_{j+1}$. To obtain the values of the function at the missing vertices, that is, for $q\in V_{j+1}{\setminus}V_j$, one applies the \emph{spectral decimation extension algorithm} 
    \begin{equation}\label{E:efct_extension}
    \phi_{j+1}(q)=\frac{(4-\lambda_j)\big(\phi_j(p)+\phi_j(r)\big)+2\phi_j(s)}{(2-\lambda_j)(5-\lambda_j)},
    \end{equation} 
    where $p,r\in V_j$ are the closest neighbors of $q$ in $V_{j+1}$, and $s\in V_j$ the furthest, c.f.~\cite[(3.2.7)]{Str06}. Figure~\ref{F:6series_ev_level2} shows the extension of $\phi_2\colon V_2\to\mathbb{R}$ to $V_3$.

\medskip

Applying~\eqref{E:efct_extension} subsequently yields a function $\phi_*\colon V_*\to\mathbb{R}$. Using the fact that eigenfunctions are continuous and $V_*$ is dense in $\SG$, one finally obtains an eigenfunction $\phi\colon \SG\to\mathbb{R}$ of the original eigenvalue $\lambda$.


\section{Main result}\label{S:main_result}
We are now in the position to precisely state and prove the main result in the paper, that is the possibility of finding sequences of high energy eigenfunctions whose associated probability distribution measures converge weakly to a Dirac at a point in $V_*{\setminus}V_0$.

\begin{theorem}\label{T:main_result}
    Let $m\geq 1$ and $q\in V_m{\setminus}V_{m-1}$. There exists a sequence $\{\psi_{n}\}_{n\geq m}$ of eigenfunctions with eigenvalues $\{\lambda_n\}_{n\geq m}$ such that 
    \begin{equation*}
        \lambda_n=\Lambda_{n+1}=5^{n-1}\Lambda_2
    \end{equation*}
    for each $n\geq m$, and
    \begin{equation}\label{E:main_result}
         \int_{\SG}\eta|\psi_{n}|^2d\mu\xrightarrow{n\to\infty}\eta(q)
    \end{equation}
    for all $\eta\in C(\SG)$.
\end{theorem}

\begin{remark}\label{R:countably_many}
    Note that there are in fact countable infinitely many sequences satisfying~\eqref{E:main_result}, and the associated eigenvalues satisfy $\lambda_n\xrightarrow{n\to\infty}\infty$.
\end{remark}

It would be interesting to decide what other measures may occur as limits of sequences of high energy eigenfunctions of the Laplacian on $\SG$, and the question will be left to future investigation.  

\medskip

The key to prove Theorem~\ref{T:main_result} is to find suitable sequences of $6$-series eigenfunctions as constructed in Section~\ref{SS:Eigenfunctions}.

\subsection{Localized profile}\label{SS:loc_profile}
To begin with, we describe what could be considered as the \emph{profile} that will serve as guide to construct the sequence of eigenfunctions associated with a given $q\in V_m{\setminus}V_{m-1}$. Throughout this section, the index $m\geq 1$ is arbitrary but fixed. 

\medskip

Recall the mappings $F_k$, $k=0,1,2$ from~\eqref{E:def_Fi} that define $\SG$ and note that the point $q$ lies at the intersection of two triangular cells, that is
    \begin{equation}\label{E:q_as_Fwjpi}
        q=F_{wk}(p_\ell)=F_{wk}(\SG)\cap F_{w\ell}(\SG)
    \end{equation}
for a unique $w\in \{0,1,2\}^{m-1}$ and $k,\ell\in\{0,1,2\}$ with $k<\ell$. Figure~\ref{F:pre-profile} illustrates an example with $m=1$, where $q=F_0(p_2)$ and $w=\text{\o}$ is the empty word (of zero length).

\begin{figure}[H]
    \includegraphics[scale=.35]{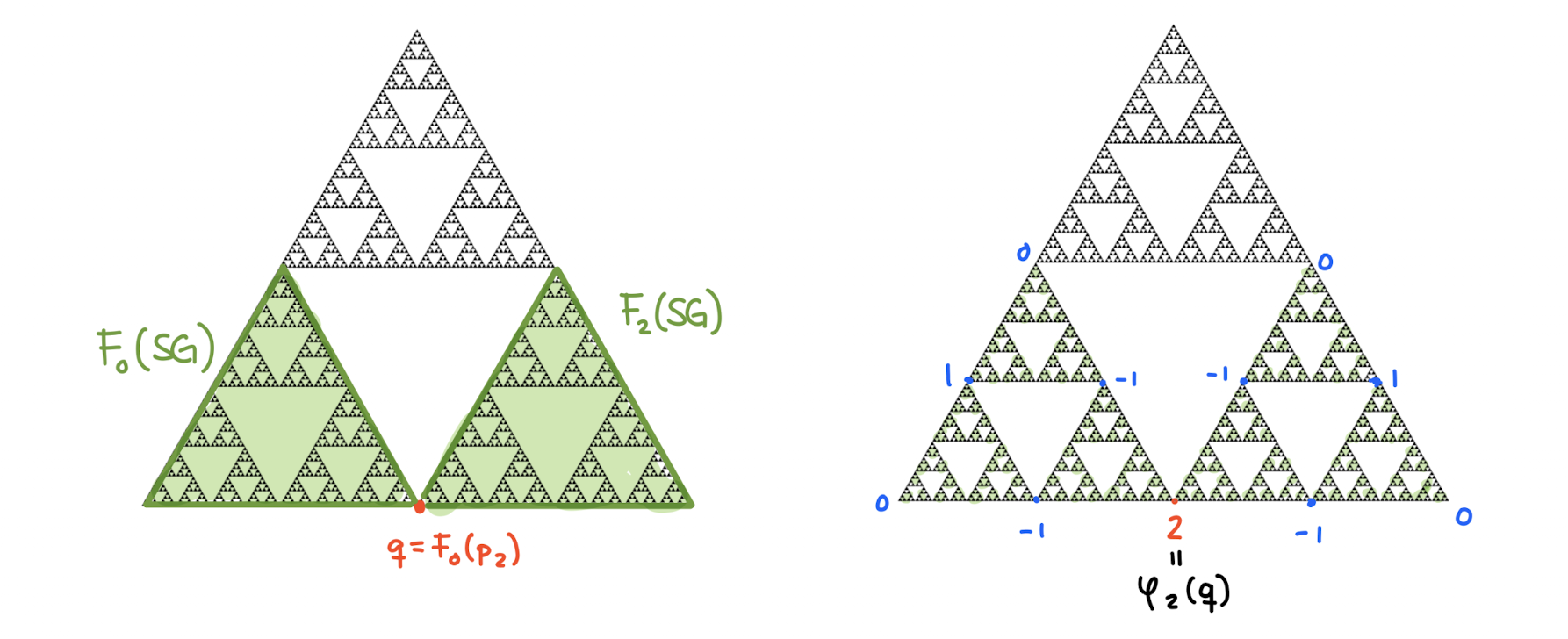}
    \caption{The ``profile'' function $\varphi_{2}$ with $q=F_0(p_2)$, $\varphi_{2}(q)=2$, and support in $F_0(\SG)\cup F_2(\SG)$.}
    \label{F:pre-profile}
\end{figure}

As profile, we will consider the $6$-series eigenfunction $\varphi_{m}\in C(\SG)$ associated with the $6$-series eigenvalue $\Lambda_{m+1}$, constructed through the procedure described in Section~\ref{SS:Eigenfunctions}. In particular, this function will satisfy
    \begin{equation}\label{E:first_efct_2_at_q}
        \varphi_{m}(q)=2,
    \end{equation}
and
    \begin{equation}\label{E:support_profile}
        {\rm supp}\,\varphi_{m}\subset F_{wk}(\SG)\cup F_{w\ell}(\SG),
    \end{equation}
see Figure~\ref{F:pre-profile}. The latter implies that this function is localized in that it is supported only on the two cells intersecting at $q$.

\subsection{Eigenfunction sequence}\label{SS:pre_sequence}
Again in this section we continue with a given $q\in V_m{\setminus}V_{m-1}$ for a fixed $m\geq 1$, expressible as the intersection~\eqref{E:q_as_Fwjpi} for a unique $w\in \{0,1,2\}^{m-1}$ and $k,\ell\in\{0,1,2\}$ with $k<\ell$. 
Using the fact that each of the mappings $F_k$ from~\eqref{E:def_Fi} has as fixed point $p_k$, for each $n\geq m$ we may write
    \begin{equation*}
    q=F_{wk}(p_\ell)=F_{wk\ell^{n-m}}(p_\ell)
     =F_{wk\ell^{n-m}}(\SG)\cap F_{w\ell k^{n-m}}(\SG).
    \end{equation*}
Again, the procedure described in Section~\ref{SS:Eigenfunctions} now allows to construct a $6$-series eigenfunction $\varphi_{n}\in C(\SG)$ associated with the $6$-series eigenvalue $\Lambda_{n+1}$, which in particular satisfies
    \begin{equation}\label{E:efcts_2_at_q}
            \varphi_{n}(q)=2
    \end{equation}
as well as the localization property
    \begin{equation}\label{E:support_efcts}
        {\rm supp}\,\varphi_{n}\subset F_{wk\ell^{n-m}}(\SG)\cup F_{w\ell k^{n-m}}(\SG).
    \end{equation}
Moreover, the symmetry of the spectral decimation algorithm~\eqref{E:efct_extension} implies that
\begin{equation}\label{E:efct_relation}
    \varphi_{n+1}|_{F_{wk\ell^{n-m+1}}(\SG)}\equiv\varphi_{n}|_{F_{wk\ell^{n-m}}(\SG)}
\end{equation}
and the same interchanging $k$ and $\ell$. As a consequence,
\begin{equation}\label{E:nested_supps}
    {\rm supp}\,\varphi_{n+1} 
    \subset{\rm supp}\,\varphi_{n}. 
\end{equation}
Figure~\ref{F:pre-eigenvectors} exemplifies $\varphi_{2}$ and $\varphi_{3}$ when $m=2$ and $q=F_0(p_2)$.

\begin{figure}[H]
    \includegraphics[scale=.35]{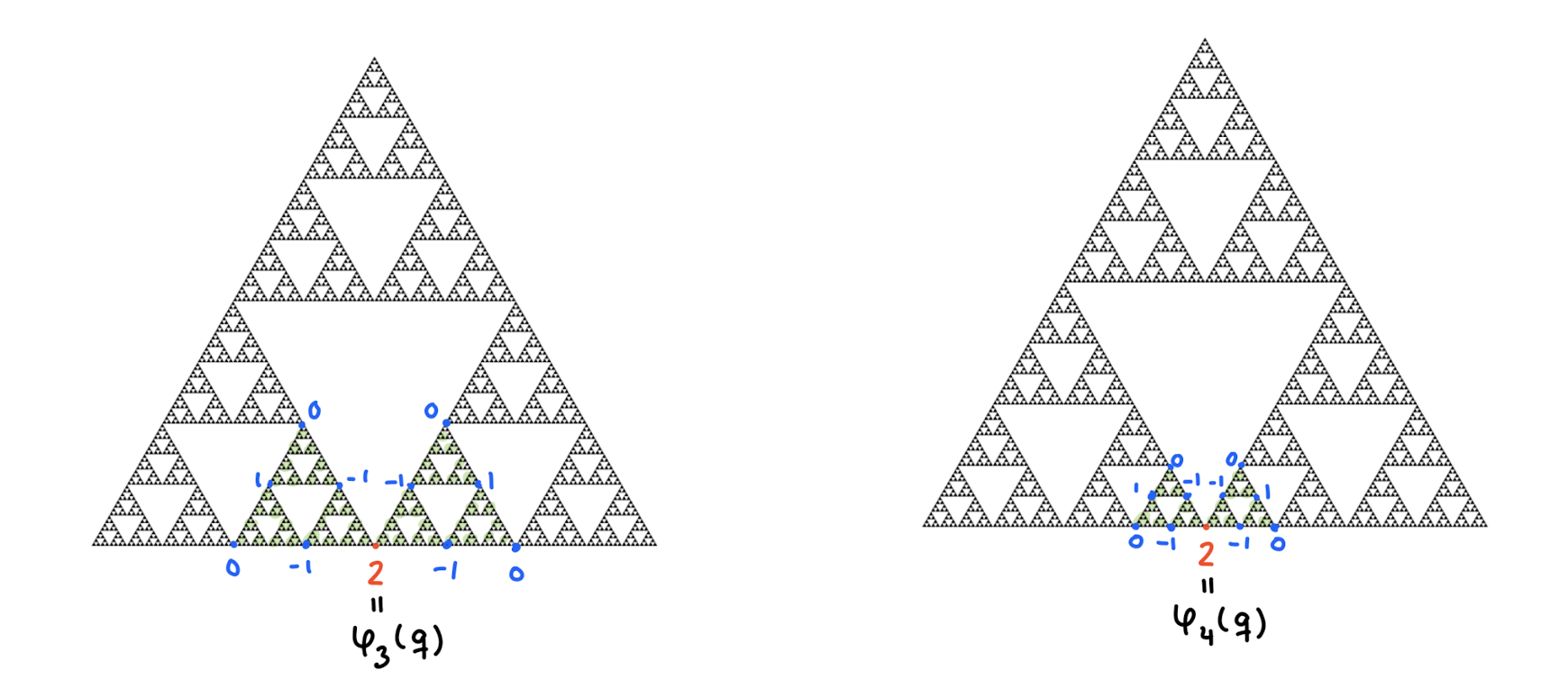}
    \caption{Eigenfunctions satisfying $\varphi_{q,n}(q)=2$ with $q=F_0(p_2)$ and their supports (colored in green).}
    \label{F:pre-eigenvectors}
\end{figure}

The relation~\eqref{E:efct_relation} provides a link between each $\varphi_{n}$, $n\geq m$, and the profile $\varphi_{m}$ which we formalize next since it will be of relevance in the proof of the main result.

\begin{proposition}\label{P:family_properties}
Let $q=F_{wk}(p_\ell)$ for some $w\in \{0,1,2\}^{m-1}$ and $k,\ell\in\{0,1,2\}$ with $k<\ell$. For each $n\geq m$, the $6$-series eigenfunction $\varphi_{n}$ satisfies
\begin{enumerate}[wide=0em,label={\rm(\roman*)},itemsep=.25em]
    \item 
    \begin{equation}\label{E:efct_m_vs_1}
        \varphi_{n}(p)=
        \begin{cases}
            \varphi_{m}(F_{wk}(\tilde{p}))&\text{if }p=F_{wk\ell^{n-m}}(\tilde{p}),\;\tilde{p}\in\SG,\\
            \varphi_{m}(F_{wj}(\tilde{p}))&\text{if }p=F_{w\ell k^{n-m}}(\tilde{p}),\;\tilde{p}\in\SG,\\
            0&\text{else},
        \end{cases}
        \qquad \forall\;p\in\SG,
    \end{equation}
    \item
    \begin{equation}\label{E:efct_m_norm}
        \|\varphi_{n}\|_{L^2_\mu(\SG)}=3^{\frac{m-n}{2}}\|\varphi_{m}\|_{L^2_\mu(\SG)}.
    \end{equation}
\end{enumerate}
\end{proposition}

\begin{proof}
    \begin{enumerate}[wide=0em,label={\rm(\roman*)},itemsep=.25em]
    \item 
    Assume for instance that $p\in F_{wk\ell^{m-n}}(\tilde{p})$ for some $\tilde{p}\in\SG$. Then, iterating~\eqref{E:efct_relation} we see that
    \begin{equation*}
       \varphi_{n}(p)=\varphi_{n}(F_{wk\ell^{n-m}}(\tilde{p}))= \varphi_{n}(F_{wk\ell^{n-m-1}}(\tilde{p}))=\cdots=\varphi_{m}(F_{wk}(\tilde{p})),
    \end{equation*}
    and analogously interchanging $k$ and $\ell$. This together with~\eqref{E:support_efcts} and~\eqref{E:nested_supps} yield~\eqref{E:efct_m_vs_1}.
    \item In view of the localization property~\eqref{E:support_efcts}, the relation~\eqref{E:efct_m_vs_1} and the scaling of the measure $\mu$ from~\eqref{E:measure_scaling}, 
    \begin{equation*}
        \begin{aligned}
            \|\varphi_{n}\|_{L^2_\mu(\SG)}^2
            &=\int_{\SG} |\varphi_{n}|^2d\mu 
            \stackrel{\eqref{E:support_efcts}}{=}\int_{F_{wk\ell^{n-m}}(\SG)}|\varphi_{n}|^2d\mu+\int_{F_{w\ell k^{n-m}}(\SG)}|\varphi_{n}|^2d\mu\\
            &=\int_{\SG}|\varphi_{n}{\circ}F_{wk\ell^{n-m}}|^2d(\mu{\circ}F_{wk\ell^{n-m}})+\int_{\SG}|\varphi_{n}{\circ}F_{wk\ell^{n-m}}|^2d(\mu{\circ}F_{w\ell k^{n-m}})\\
            &\stackrel{\eqref{E:efct_m_vs_1}}{=}\int_{\SG}|\varphi_{m}{\circ}F_{wk}|^2d(\mu{\circ}F_{wk\ell^{n-m}})+\int_{\SG}|\varphi_{m}{\circ}F_{w\ell}|^2d(\mu{\circ}F_{w\ell k^{n-m}})\\
            &\stackrel{\eqref{E:measure_scaling}}{=}3^{m-n}\int_{\SG}|\varphi_{m}{\circ}F_{wk}|^2d(\mu{\circ}F_{wk})+3^{m-n}\int_{\SG}|\varphi_{m}{\circ}F_{w\ell}|^2d(\mu{\circ}F_{w\ell})\\
            &=3^{m-n}\int_{F_{wk}(\SG)}|\varphi_{m}|^2d\mu+3^{m-n}\int_{F_{w\ell}(\SG)}|\varphi_{m}|^2d\mu\\
            &=3^{m-n}\int_{\SG} |\varphi_{m}|^2d\mu.
        \end{aligned}
    \end{equation*}
    \end{enumerate}
\end{proof}


\subsection{Proof of the main result}
\begin{proof}[Proof of~\Cref{T:main_result}]
    Let $q\in V_m{\setminus}V_{m-1}$ with fixed $m\geq 1$, expressed as $F_{wk}(p_\ell)$ for some $w\in \{0,1,2\}^{m-1}$ and $k,\ell\in\{0,1,2\}$ with $k<\ell$. Consider the sequence of normalized eigenfunctions $\{\psi_{n}\}_{n\geq m}$ given by
    \begin{equation}\label{E:def_psi_n}
        \psi_{n}:=\frac{\varphi_{n}}{\|\varphi_{n}\|_{L^2_\mu(\SG)}},
    \end{equation}
    where $\varphi_{n}$ denote the eigenfunctions constructed in Section~\ref{SS:pre_sequence}.
    Let now $\eta\in C(K)$. 
    By virtue of Proposition~\ref{P:family_properties}, for any $n\geq m$ it holds that
    \begin{align}\label{E:main_h_01}
        \int_{\SG}\eta|\psi_{n}|^2\,d\mu
        &\stackrel{\eqref{E:support_efcts}}{=}\int_{F_{wk\ell^{n-m}}(\SG)}\eta|\psi_{n}|^2d\mu+\int_{F_{w\ell k^{n-m}}(\SG)}\eta|\psi_{n}|^2d\mu\notag\\
        &\stackrel{\eqref{E:def_psi_n}}{=}\frac{3^{n-m}}{\|\varphi_{m}\|_{L^2_\mu(\SG)}^2}\int_{\SG}\eta{\circ}F_{wk\ell^{n-m}}|\varphi_{m}{\circ}F_{wk}|^2d(\mu{\circ}F_{wk\ell^{n-m}})\notag\\
        &+\frac{3^{n-m}}{\|\varphi_{m}\|_{L^2_\mu(\SG)}^2}\int_{\SG}\eta{\circ}F_{w\ell k^{n-m}}|\varphi_{m}{\circ}F_{w\ell}|^2d(\mu{\circ}F_{w\ell k^{n-m}})\notag\\
        &\stackrel{\eqref{E:measure_scaling}}{=}\frac{1}{\|\varphi_{m}\|_{L^2_\mu(\SG)}^2}\int_{\SG}\eta{\circ}F_{wk\ell^{n-m}}|\varphi_{m}{\circ}F_{wk}|^2d(\mu{\circ}F_{wk})\notag\\
        &+\frac{1}{\|\varphi_{m}\|_{L^2_\mu(\SG)}^2}\int_{\SG}\eta{\circ}F_{w\ell k^{n-m}}|\varphi_{m}{\circ}F_{w\ell}|^2d(\mu{\circ}F_{w\ell})\notag\\
        &=\int_{F_{wk}(\SG)}(\eta{\circ}F_{wk\ell^{n-m}}{\circ}F_{wk}^{-1})|\psi_{m}|^2d\mu\notag\\
        &+\int_{F_{wj}(\SG)}(\eta{\circ}F_{w\ell k^{n-1}}{\circ}F_{w\ell}^{-1})|\psi_{m}|^2d\mu.
    \end{align}
    Note now that $wk\in\{0,1,2\}^m$, whence by definition $F_{wk}$ is a contraction of ratio $3^{-{m}}$, and in addition $F_\ell$ has as fixed point $p_\ell$. Thus, for any $\tilde{p}\in\SG$,
    \begin{equation*}
        \begin{aligned}
        \big\|F_{wk\ell^{n-m}}(\tilde{p})-q\big\|
        &=\big\|F_{wk\ell^{n-m}}(\tilde{p})-F_{wk}(p_\ell)\big\|\\
        &\leq 3^{-{m}}\big\|F_{\ell^{n-m}}(\tilde{p})-p_\ell\big\|\leq \big\|F_{\ell^{n-m}}(\tilde{p})-p_\ell\big\| 
        \xrightarrow{n\to\infty}0.
        \end{aligned}
    \end{equation*}
    Since $\eta$ is continuous and $\SG$ compact, applying dominated convergence to~\eqref{E:main_h_01} and taking on account~\eqref{E:support_efcts} finally yields
    \begin{equation*}
        \begin{aligned}
            \lim_{n\to\infty}\int_{\SG}\eta|\psi_{n}|^2d\mu
            &=\int_{F_{wk}(\SG)}\lim_{n\to\infty}\eta\big(F_{wk\ell^{n-m}}(F_{wk}^{-1}(\tilde{p})\big)|\psi_{m}(\tilde{p})|^2d\mu(\tilde{p})\\
            &+\int_{F_{w\ell}(\SG)}\lim_{n\to\infty}\eta\big(F_{w\ell k^{n-1}}(F_{w\ell}^{-1}(\tilde{p})\big)|\psi_{m}(\tilde{p})|^2d\mu(\tilde{p})\\
            &=\eta(q)\int_{F_{wk}(\SG)\cup F_{w\ell}(\SG)}|\psi_{m}|^2d\mu=\eta(q)\|\psi_{m}\|_{L^2(\SG)}^2
            =\eta(q)
        \end{aligned}
    \end{equation*}
    as we wanted to prove.
\end{proof}

\vspace*{3em}

	\subsection*{Acknowledgments}
	This work is partially supported by the ANR project Smooth ANR-22CE40-0017.  The second author is grateful to Nicolas Burq, Patrick G\' erard and Chenmin Sun for helpful discussions on quantum limits.

\bibliographystyle{amsplain}
\bibliography{QE_SG_refs}
\end{document}